\documentclass[11pt,a4paper]{article}
\usepackage{amsmath}
\usepackage{amsfonts}
\usepackage{amssymb}
\usepackage{mathrsfs}
\usepackage{subfigure}
\usepackage{epsfig}
\usepackage{pstricks}
\usepackage{psfrag}

\newtheorem{prop}{Proposition}
\newtheorem{definition}{Definition}
\newtheorem{corollary}{Corollary}
\newtheorem{example}{Example}
\newenvironment{proof}{
    {\bf Proof}}{\hbox{\ }\hfill$|||$
}

\newcommand{\hlf}{\frac{1}{2}}

\newcommand{\figref}[1]{Fig.~\ref{#1}}
\newcommand{\exref}[1]{Example~\ref{#1}}
\newcommand{\corref}[1]{Corollary~\ref{#1}}

\newcommand{\secref}[1]{Section~\ref{#1}}

\newcommand{\propref}[1]{Proposition~\ref{#1}}

\newcommand{\smt}[2]{\Bigl[
 \begin{smallmatrix} #1\\ #2
 \end{smallmatrix}\Bigr]}

\newcommand{\ga}{\alpha}

\newcommand{\gs}{\sigma} 
\newcommand{\eps}{\epsilon}
\newcommand{\gt}{\tau}
\newcommand{\Chi}{\chi}

\newcommand{\bA}{\mathbf{A}}   
\newcommand{\bD}{\mathbf{D}}   
\newcommand{\bI}{\mathbf{I}}   
\newcommand{\bJ}{\mathbf{J}}   
\newcommand{\be}{\mathbf{e}}   
   
\newcommand{\bj}{\mathbf{j}}   
   
\newcommand{\br}{\mathbf{n}}   

\newcommand{\R}{\mathbb{R}}

\newcommand{\AL}{{\mathcal{A}}}
\newcommand{\BL}{{\mathcal{B}}}
\newcommand{\DL}{{\mathcal{D}}}
\newcommand{\EL}{{\mathcal{E}}}

\newcommand{\ifn}{indicator function}
\newcommand{\sT}{T}    
\newcommand{\Tbase}{\sT}   
\newcommand{\sH}{{ \cal{H}}}
\newcommand{\supH}{c}
\newcommand{\pln}{{\mathcal{P}}}
\newcommand{\ed}{e}
\newcommand{\fc}{f}

\newcommand{\Dim}{n}
\newcommand{\cl}{c}  
\newcommand{\clp}{\cl'}  
\newcommand{\cls}{\cl^1}  
\newcommand{\Oc}{Overcomplete}  %
\newcommand{\oc}{overcomplete}  %
\newcommand{\sd}{scaled-down}  
\newcommand{\straddl}{straddl}  %
\newcommand{\pt}{\mathbf{p}}  
\newcommand{\nr}{\mathbf{n}}  
\newcommand{\spl}{s}  
\newcommand{\sit}{shift-invariant tessellation}

\begin{document}%
\title{Refinability of splines derived from
regular tessellations}

\author{J\"org Peters}

\maketitle


\begin{abstract}
Splines can be constructed by convolving the \ifn\ of 
a cell whose shifts tessellate $\R^\Dim$.
This paper presents simple, non-algebraic
criteria that imply that, for regular \sit s,
only a small subset of such spline families yield nested spaces:
primarily the well-known tensor-product and box splines.
Among the many non-refinable constructions are hex-splines and
their generalization to the Voronoi cells of non-Cartesian root lattices.
\end{abstract}

\section{Introduction}
Univariate uniform B-splines can be defined by repeated convolution, 
starting with the \ifn s\footnote{An \ifn\ takes on the value one 
on the interval but is zero otherwise.}
of the intervals or cells delineated by knots.
This construction implies local support and 
delivers a number of desirable properties (see \cite{Boor:1978:PGS,deboor87e})
that have made B-splines the representation of choice 
in modeling and analysis. In particular, B-splines are
refinable. That is, they can be exactly represented as a linear combinations
of B-splines with a finer knot sequence. Refinability 
is a key ingredient of multi-resolution and
adaptive and sparse representation of data. Refinability also guarantees
monotone decay of error when shrinking the intervals.
 
\def\swid{0.25\linewidth}
\def\wid{0.30\linewidth}
\begin{figure}[h]
\centering
   \subfigure[natural and man-made hex-tilings 
]{
   \includegraphics[width=\wid]{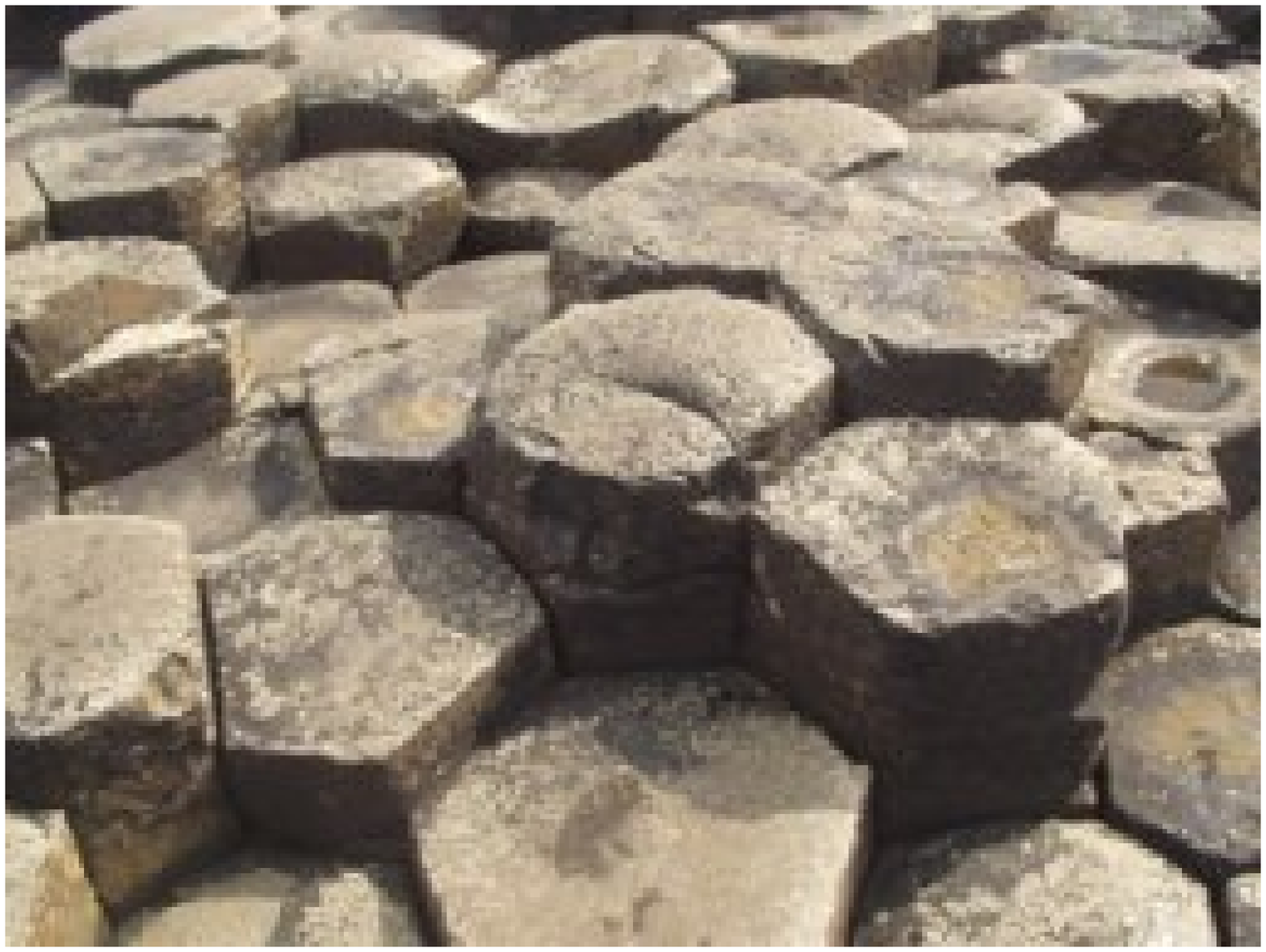}
   \includegraphics[width=\swid]{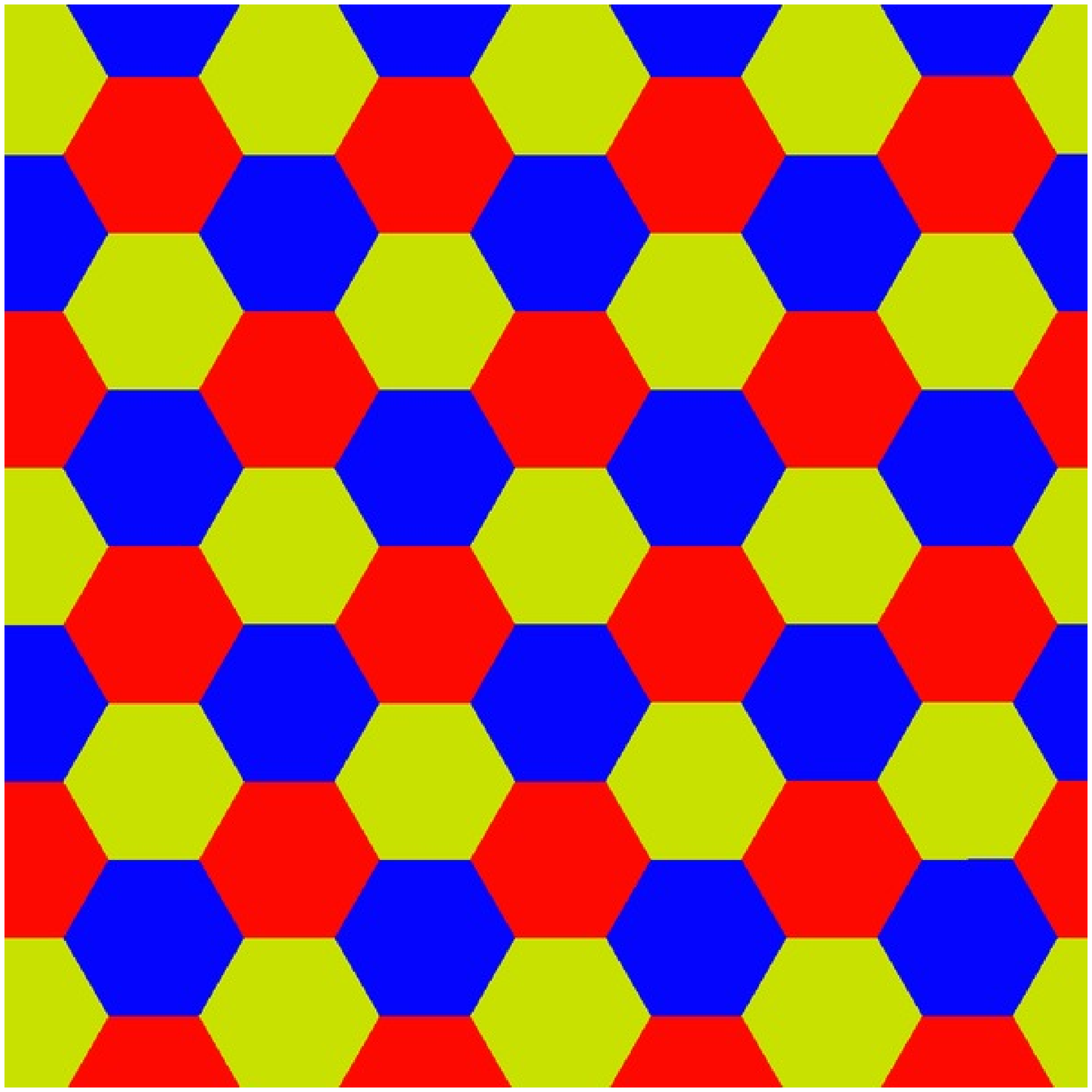}}
   \subfigure[half-scaled hex-tiling]{
\psset{unit=0.6cm} \pspicture(-3,-3)(3,3)
\pspolygon[fillstyle=solid,fillcolor=lightgray]
   (2,0)(1,-1.732)(-1,-1.732)(-2,0)(-1,1.732)(1,1.732)
\rput[l](0,-0.5){$H$}
\pspolygon[fillstyle=none]
   (1,0)(0.5,-0.866)(-0.5,-0.866)(-1,0)(-0.5,0.866)(0.5,0.866)
\pspolygon[fillstyle=none]
   (1,1.732) (0.5,2.598)(-0.5,2.598)(-1,1.732)(-0.5,0.866)(0.5,0.866)
\pspolygon[fillstyle=none]
   (-.5,0.866)(-1,0) (-2,0)(-2.5,0.866)(-2,1.732)(-1,1.732)
\pspolygon[fillstyle=none]
   (-.5,-0.866)(-1,0) (-2,0)(-2.5,-0.866)(-2,-1.732)(-1,-1.732)
\pspolygon[fillstyle=none]
   (1,-1.732) (0.5,-2.598)(-0.5,-2.598)(-1,-1.732)(-0.5,-0.866)(0.5,-0.866)
\pspolygon[fillstyle=none]
   (.5,-0.866)(1,0) (2,0)(2.5,-0.866)(2,-1.732)(1,-1.732)
\pspolygon[fillstyle=none]
   (.5,0.866)(1,0) (2,0)(2.5,0.866)(2,1.732)(1,1.732)
\endpspicture
} 
\caption{Hexagonal tessellations. 
(a) basalt formation 
 and tiles ({http://en.wikipedia.org/wiki/\{Basalt,Hexagonal\_tiling\}})
(b) Non-nesting of the hex partition in Example 1.
}
\label{fig:petal}
\end{figure}
By tensoring univariate B-splines, we can obtain splines
on Cartesian grids in any dimension.
Box-splines \cite{deboor93box} generalize tensoring 
by allowing convolution in directions other than orthogonal ones.
%
As a prominent example in two variables, 
the linear 3-direction box-spline consists
of linear pieces over each of six equilateral triangles surrounding 
one vertex.
Shifts of this `hat function' on an equilateral triangulation sum to one. 
Convolution of the hat function with itself results in a twice
continuously differentiable function of degree 4;
and $m$-fold convolution yields a function of degree $3m-2$
with smoothness $C^{2m}$.
Since this progression skips odd orders of smoothness,
van der Ville et al. \cite{van:04} proposed to directly
convolve the \ifn\ of the hexagon and
build splines customized to the hexagonal tessellation of the plane
(cf. \figref{fig:petal}a). 
They went on to show that the resulting hex-splines share a number 
of desirable properties familiar from box-splines. But the authors
did not settle whether the splines were \emph{refinable}
\cite{perscomBanff},
i.e.\ whether hex-splines of the given hexagonal tessellation
$\Tbase$ can be represented as linear combinations of hex-splines based
on a \sd\ hexagonal tessellation, say $\hlf\Tbase$.
Generalizing the analysis of hex-splines, 
\begin{itemize}
\item
this paper presents simple non-algebraic criteria necessary for 
regular \sit s to admit refinable \ifn s.
\end{itemize}
For example, such a tessellation must contain, for every cell facet $f$, 
the plane through $f$.
Therefore, requiring refinability, even of just the constant spline,
strongly restricts allowable tessellations.
\begin{itemize}
\item
In contrast to tensor-product and box splines,
we show that hex-splines and similar constructions 
can only be \emph{scaled, but not refined}: 
scaled hex-spline spaces are not nested.
\item
The analysis extends to \oc\ families (superpositions) of spline spaces.
\end{itemize}
The following example illustrates how non-refinability leads to
loss of monotonicity of the approximation error under scaling:
for one  or more steps \emph{halving the scale can increase the error}.
By contrast, nested spaces guarantee monotonically decreasing error.

\begin{example} {\rm
\label{ex:error}
Let $\sH^i$ be the space of \ifn s over a
regular tessellation by hexagons of diameter $2^{-i}$ 
and such that, at each level of scaling, the origin
is the center of one hexagon. 
Denote by $H$ the \ifn\ in $\sH^0$ whose support hexagon is
centered at the origin.
$\sH^1$ does not contain a linear combination of functions
that can replicate $H$ 
since the supports of the six relevant scaled \ifn s are bisected by
the boundary of the support of $H$ (see \figref{fig:petal}b).
Correspondingly, the $L^2$ approximation error to $H$ from $\sH^1$ is 
$\frac{6}{2}A^1 > 0$ where $A^1$ is the area of the 
hexagon with diameter $\hlf$. Since the error from $\sH^0$ is
by construction zero, the scaling by $1/2$ 
has increased the error. 
By carefully adding to $H$ an increasing number of \sd\ copies,
small increases in the $L^2$ error can be distributed 
over multiple consecutive steps. 
} \hfill$\Box$
\end{example}

{\bf Overview.}
\secref{sec:lit} reviews tessellations induced by lattices,
hex-splines and their generalizations.
\secref{sec:nonref} exhibits two 
non-algebraic criteria, chosen for their simplicity,
for testing whether a tessellation can support a refinable
space of splines that are constructed by convolution of \ifn s of its cells.
\secref{sec:overcomplete} extends this investigation to a
multiple covering of $\R^\Dim$ by distinct families of \ifn s.

\section{Splines from lattice Voronoi cells}
\label{sec:lit}
A $\Dim$-dimensional lattice is a discrete subgroup of full rank in
a $\Dim$-dimensional Euclidean vector space.
Alternatively, such a lattice may be viewed as inducing a tessellation
of space into identical cells without $\Dim$-dimensional overlap\footnote{
A common convention is to define the cells to be half-open sets 
so that they do not overlap on facets, but nevertheless cover.}.
The tessellation is then generated by the translational shifts
of one cell. For example, 
lattice points can serve as sites of Voronoi cells.
The Euclidean plane admits three highly symmetric \sit s:
partition into equilateral triangles, squares, or hexagons respectively.
Repeated convolution starting with the \ifn\ of any of these 
polygons yields spline functions of local support and increasing degree.
The regular partition into squares 
gives rise to uniform tensor-product B-splines and
the regular triangulation and its hexagonal dual to box splines. 

An interesting additional type of spline arises 
from convolving the \ifn\ $H$ of the hexagon with itself.
%
Such hex-splines, a family of $C^{k-1}$ splines supported on 
a local $k+1$-neighborhood, were developed and analyzed by 
van De Ville et al. \cite{van:04}.
That paper compares hex-splines to tensor-product splines
and uses the Fourier transform of hex-splines to derive, 
for low frequencies, the $L^2$ approximation order, as a combination of 
the projection into the hex-spline space and a quasi-interpolation error.
\cite{condat:05} derived quasi-interpolation formulas 
and showed promising results when applying
hex-splines to the reconstruction of images
(see also \cite{Condat:2006:ERH,Condat:2007:QIS,Condat:2008:NOS}).
Van De Ville et al. \cite{van:04}.
also observed that hexagons are Voronoi cells 
of a lattice and that the cell can
be split into three quadrilaterals, using one of two choices of the 
central split. Thus $H$ can be split into three constant box splines
whose mixed convolution yields higher-order splines
\cite{Kimperscom,journals/tsp/MirzargarE10}. 
However, while box-splines are refinable, we will see that
hex-splines are not refinable in a shift-invariant way.


\section{Refinability constraints}
\label{sec:nonref} 
We consider a polyhedral tessellation $\Tbase$ of $\R^\Dim$ into 
unpartitioned $\Dim$-dimensional units, called \emph{cells}, that 
are bounded by a finite number of $\Dim-1$-dimensional facets.
We denote by $\Chi(\Tbase)$ the space of \ifn s of the cells of $\Tbase$
and by $\Chi(T^{1})$ the space of \ifn s
on some \sd\ copy $T^{1}$ of $\Tbase$.
The space $\Chi(\Tbase)$ is \emph{refinable} if
each \ifn\ in $\Chi(\Tbase)$
can be represented as a linear combination of functions in $\Chi(T^{1})$.

Establishing whether a tessellation $\Tbase$ admits a refinable space of 
\ifn s therefore requires proving the existence of
weights such that a linear combination
of elements in $\Chi(T^{1})$ with these weights reproduces
each element in $\Chi(\Tbase)$. 
\propref{prop:straddle} below provides
a much simpler \emph{necessary condition} that avoids such algebraic analysis.
While our focus is on \sit s, \propref{prop:straddle}
applies more generally and also to cell boundaries of co-dimension 
greater than 1.
Its proof uses the notion of a $\cl^1\in T^1$
\straddl ing a facet of a cell $\cl \in T$. 
A cell $\cl^1$ \emph{\straddl es} a facet $f$ of $\cl$
if there exists a point $\pt$ on $f$, a unit vector $\nr$ 
normal to $f$ at $\pt$ and $\eps > 0$ such that both $\pt+\eps \nr \in \cl^1$
and $\pt-\eps \nr \in \cl^1$.

\begin{prop}
Let $\Tbase$ be a polyhedral tessellation of $\R^\Dim$
and $T^{1}$ its \sd\ copy. Then  $\Chi(\Tbase)$ is
refinable only if every facet of $\Tbase$ is the 
union of facets of $T^{1}$.
\label{prop:straddle}
\end{prop}

\begin{proof}
Assume that a facet $f$ of a cell $\cl$ in $\Tbase$ is not 
a union of facets of $T^{1}$.
Then, since $T^1$ is a tessellation, 
some cell $\cls$ of $T^{1}$ must \straddl e $f$.
Let $H^1 \in \Chi(T^{1})$ be the \ifn\ of $\cls$
and $H$ the \ifn\ of $\cl$.
Then, in order to reproduce the unit step of $H$ across $f$,
$H^1$ must simultaneously take on both the value $0$ and the value $1$.
\end{proof}

%

Translation-invariant or \sit s are a special case of transitive tilings
where every cell can be mapped to every other cell 
by translation, without rotation.

\begin{prop}
If $T$ is a \sit, $\Chi(T)$ is refinable only if, 
$T$ contains, for each facet $\fc$, the hyperplane through $\fc$. 
\label{prop:ext}
\end{prop}
\begin{proof}
The coarser-scaled copies of $T$ contain enlarged copies of every facet in $T$.
By \propref{prop:straddle} these copies must be a union of facets of $\Tbase$.
Therefore a \emph{shifted copy} of every facet is strictly contained
in the interior of and so extended by some coarser facet. 
Shift-invariance then implies that \emph{every} facet $f$ lies 
strictly inside such an extension. Ever coarser tessellations provide
a sequence of extensions of $f$ in any direction by any amount.
\end{proof}

By inspection of the three regular tessellations of the plane,
only the Cartesian grid and the uniform triangulation satisfy
\propref{prop:ext}, but not the partition into hexagons. 
\begin{corollary}
Hex splines are not refinable.
\end{corollary}
We can generalize this observation by simplifying the inspection 
criterion.

\def\swid{0.35\linewidth}
\begin{figure}[h]
   \centering
   \psfrag{f1}{$f_1$}
   \psfrag{n1}[r]{$\nr_1$}
   \psfrag{f1p}{$f'_1$}
   \psfrag{f2}{$f_2$}
   \psfrag{n2}{$\nr_2$}
   \psfrag{A}{$\cl$}
   \psfrag{B}{$\clp$}
   \psfrag{sym}{$F_2$}
   \psfrag{int}{$F_1$}
   \includegraphics[width=\swid]{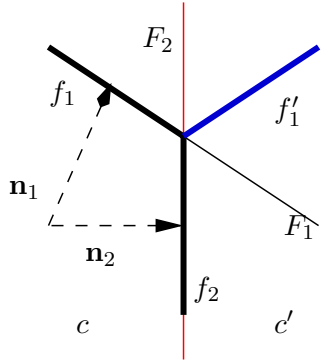}
   \caption{ A pair of facets $\fc_1$, $\fc_2$ of a cell $c$ meet with an
   obtuse angle: the
   outward pointing normals $\nr_1$, $\nr_2$ (dashed) have a strictly 
     positive inner product. 
   }
   \label{fig:obtuse}
\end{figure}
We say that two abutting facets $f_1$ and $f_2$ of a cell $\supH$
\emph{meet with an obtuse angle} if, for $i=1,2$,
there exist unit vectors $\nr_i$ orthogonal to $f_i$ and
outward pointing so that $\nr_1\cdot\nr_2>0$.
\begin{prop}
\label{prop:obtuse}
Let $\Tbase$ be a tessellation of $\R^\Dim$ by shifts of one
polyhedral cell $\cl$.
If two facets $\fc_1$ and $\fc_2$ of $\cl$ meet with an obtuse angle 
and if
$\clp$, the reflection of $\cl$ across $\fc_2$, is a cell of $\Tbase$ 
then $\Chi(\Tbase)$ is not refinable.
\end{prop}
\begin{proof}
Assume $\Chi(\Tbase)$ is refinable under the given conditions.
Let $\clp$ be the reflection of $\cl$ across (the plane through)
$\fc_2$.  Since $\clp$ must not overlap $\cl$,
obtuse angles exceeding $\pi$, such as the reentrant corner of
an L-shaped cell, cannot occur in $\cl$.
Denote by  $\fc'_1$ the reflection of $\fc_1$ across $\fc_2$ and
by $e$ the common intersection of $f_1$, $f_2$ and $f'_1$ 
(see \figref{fig:obtuse}). 

Within $\clp$, by reflection, the facets $\fc_2$ and $\fc'_1$ meet at 
$\ed$ with an obtuse angle. 
By \propref{prop:ext} the extension $F_1$ of $\fc_1$ lies in $\Tbase$.
Since the outward-pointing normal of $\fc_2$ with respect to
$\clp$ is $-\nr_2$,
$\fc_2$ and $F_1$ meet at $\ed$ with an acute angle.
Therefore $F_1$ extends $\fc_1$ into and splits $\clp$.
This contradicts the definition of a cell as an unpartitioned unit
and hence the initial assumption.
\end{proof}

The next \propref{prop:obtuse} allows us to quickly decide 
which of the (symmetric crystallographic) root lattices 
$\AL_n$, $\AL^*_n$, 
$\BL_n$, 
$\DL_n$, $\DL^*_n$, 
$\EL_j$, $j=6,7,8$ \cite{conway98}
are suitable for building refinable splines by convolution of
their nearest-neighbor (Voronoi) cells.
\begin{corollary}
Splines obtained by convolving the Voronoi cell of a
non-Cartesian crystallographic root lattice are not refinable.
\label{cor:voronoi}
\end{corollary}
\begin{proof}
We test whether the Voronoi cells of the root lattices contain 
a pair of abutting faces that meet with an obtuse angle. 
We may assume that one Voronoi site (cell center) is at the origin.
By definition of a Voronoi cell, 
the position vectors $\nr_1$ and $\nr_2$ of two adjacent nearest neighbors,
as identified by their root system,
are normal to the corresponding abutting bisector facets.
Therefore these facets meet with an obtuse angle if $\nr_1\cdot\ \nr_2 > 0$.

The $\AL_n$ lattice is traditionally defined via an embedding 
in $\R^{\Dim+1}$, $\Dim>1$. 
More convenient for our purpose is the alternative 
geometric construction in $\R^\Dim$ via the 
$n\times n$ generator matrix $\bA_n := \bI_n + \frac{c_n}{n}\bJ_n$
of Theorem 1 of \cite{kim:2010:andual}.
Here $\bI_n$ is the identity matrix, $\bJ_n$ the $n\times n$
matrix of ones and $c_n := \frac{-1+\sqrt{n+1}}{n}$. 
Denoting the $i$th coordinate vector by $\be_i$,
we choose $\be_1$ and $\be_1+\be_2$ on 
the Cartesian grid, and map them via $\bA_n$ to the 
nearest $\AL_n$ neighbors of the origin. 
The inner product of the images of $\be_1$ and $\be_1+\be_2$ is
\begin{equation*}
   \bA_n\be_1 \cdot \bA_n(\be_1+\be_2) 
   =
   \frac{n+4c_n+c_n^2}{n}
   =
   \frac{2}{n}(n+\sqrt{n+1} - 1) > 0.
\end{equation*}
For the $\AL^*_n$ lattice, the computation is identical except that 
$c_n := \frac{-1+\frac{1}{\sqrt{n+1}}}{n}$.
The inner product is 
$\frac{1}{n(n+1)}(n^2 -2n-2+2\sqrt{n+1})) > 0$.

For the $\DL_n$ lattice, defined in $n\ge 3$ dimensions, 
the generator matrix is
$
   \bD_n :=
   \begin{bmatrix}
   \bI_{n-1} & - \be^{n-1}_{n-1} \\
   -\bj^t_{n-1} & - 1 \\
   \end{bmatrix}
$
(see e.g.\  Section 7 of \cite{Kim:2011:SBS})
and
\begin{equation*}
   \bD_n\be_1 \cdot \bD_n(\be_1+\be_2) 
   = 3 > 0.
\end{equation*}
Since $\bD^{-t}_n$ is the generator of $\DL^*_n$, the inner 
product for $\DL^*_n$ is  $2$.

For $\BL_n$, the Cartesian cube lattice has an inner product of $0$
identifying its uniform tensor-product B-spline constructions 
as potentially refinable
(which indeed they are). On the other hand,
splitting each cube by adding the diagonal directions
of the full root system
\cite{Kim:SBL:2008} yields the inner product 
$\be_1 \cdot \bj = 1$.

For $\EL_6$, we select the root vectors 
$(1,1,0,0,0,0) $ and $(1,1,1,1,1,\sqrt{3})/2$ with inner product $1$.
For $\EL_7$, we select the root vectors 
$(1,1,0,0,0,0,0) $ and $(1,1,1,1,1,1,\sqrt{2})/2$ with inner product $1$.
For $\EL_8$, we select the root vectors 
$(1,1,0,0,0,0,0,0) $ and $\bj_8/2$ with inner product $1$.
\end{proof}

The equilateral triangulation in $\R^2$
is dual to the `honeycomb lattice' which is not a standard lattice.
The equilateral triangulation yields an inner product of 
$\frac{-1}{2}$ compatible with refinability
and indeed plays host to the refinable `hat' function.

\section{\Oc\ spaces}
\label{sec:overcomplete}
Since the evaluation of hex-splines by convolving 
three families of box splines already makes use of a large number of
terms, it is reasonable to investigate whether superposition 
of several families of hex-splines 
are refinable as a family.
That is, we consider a family of distinct \sit s $\{\Tbase_j\}_{j=0..J}$ 
obtained by shifts of $\Tbase_0$.
Their union covers $\R^\Dim$ $J+1$-fold.
We check refinability of the family, i.e.\
whether each member of the family can be expressed
as a linear combination of the \sd\ copies of splines of
the family. \exref{ex:hexref} makes this concrete for $J=2$.
\medskip

\begin{figure}[!ht]
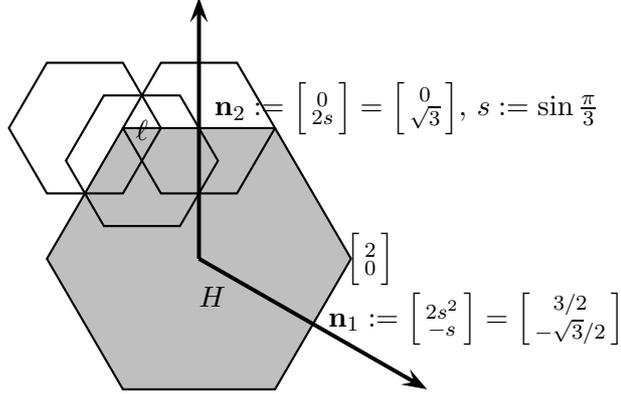

\centering
\psset{unit=1.0cm} \pspicture(-3,-2)(3,3)
\pspolygon[fillstyle=solid,fillcolor=lightgray](2,0)(1,-1.732)(-1,-1.732)(-2,0)(-1,1.732)(1,1.732)
\def\ss{0.8660254040}
\def\s4{0.8660254040}
\psline[linewidth=1.5pt,arrowscale=1.5]{->}(0,0)(0,3.464) 
\psline[linewidth=1.5pt,arrowscale=1.5]{->}(0,0)(3,-1.732) 
\rput[l](0.2,2){$\br_2 := \smt{0}{2s} = \smt{0}{\sqrt{3}}$,\ 
$s := \sin \frac{\pi}{3}$}
\rput[l](1.7,-0.8){$\br_1 := \smt{2s^2}{-s} = \smt{3/2}{-\sqrt{3}/2}$}
\rput[l](0,-0.5){$H$}
\pspolygon[fillstyle=none](0.5,2.598)(-0.5,2.598)(-1,1.732)(-0.5,0.866)(0.5,0.866)(1,1.732)
\pspolygon[fillstyle=none](-1.0,2.598)(-2,2.598)(-2.5,1.732)(-2,0.866)(-1,0.866)(-.5,1.732)
\pspolygon[fillstyle=none](-.25,2.165)(-1.25,2.165)(-1.75,1.3)(-1.25,0.433)(-.25,0.433)(0.25,1.299)
\rput(-0.75,1.7){$\ell$} 
\rput(2.25,0){$\smt{2}{0}$} 
\endpspicture
\caption{ A lozenge-shaped pair of triangles $\ell$
is in the common support of three half-scaled translated copies
of the grey hexagon $h$.
}
\label{fig:hexref}
\end{figure}

\begin{example}
\label{ex:hexref}
{\rm
Denote by $\Tbase_0$ a tessellation of the plane into unit-sized hexagons
and by $\Tbase_1$ and $\Tbase_2$ its shifts by integer multiples of
$\hlf\br_1$ and $\hlf\br_2$
(see \figref{fig:hexref}).
Let $H \in \Chi(\Tbase_0)$ be the \ifn\ of the unit hexagon $h$
centered at the origin.
Consider the three $\hlf$-scaled, translated copies of $h$ shown
in \figref{fig:hexref}.
The three copies intersect in a lozenge-shaped pair of triangles $\ell$.
No other shifts of the $\hlf$-scaled hexagons in $T_0$, $T_1$ or $T_2$
overlap $\ell$.
Therefore
any linear combination of \ifn s in $\{\Tbase_j\}_{j=0..J}$
has a single value on $\ell$. 
Since $\ell$ straddles the boundary of $h$, 
this constant linear combination would have to be simultaneously 1 and 0
to replicate $H$.
\hfill $\Box$}
\end{example}

\exref{ex:hexref} suggests the following generalization of \propref{prop:ext}.
A \emph{superposition} $T$ of a family of \sit s
$\{\Tbase_j\}_{j=0..J}$ of $\R^\Dim$ is the tessellation obtained 
by partitioning $\R^\Dim$ by all cell facets in any of the
tessellations $\Tbase_j$. $T$ differs from $\{\Tbase_j\}_{j=0..J}$
in that it contains fractions or \emph{pieces} of the original cells.

\begin{prop}
Given a family $\{\Tbase_j\}_{j=0..J}$ of polyhedral \sit s of $\R^\Dim$,
the space of \ifn s $\bigcup_{j=0..J}\Chi(\Tbase_j)$ 
is refinable only if the superposition $T$ of $\{\Tbase_j\}_{j=0..J}$
contains, for each facet $f$, the hyperplane through $f$. 
\label{prop:family}
\end{prop}
\begin{proof}
Assume that $\bigcup_{j=0..J}\Chi(\Tbase_j)$ is refinable.
Assume additionally that a \sd\ copy $T^1$ of $T$ contains an
unpartitioned piece $c^1 \in T^1$ 
that \straddl es a facet $\fc$ of some cell $\cl$ in one of the $\Tbase_j$.
Analoguous to the proof of \propref{prop:straddle},
any linear combination of the \ifn s of the family $\{\Tbase_j\}_{j=0..J}$
that replicates the step of the \ifn\ on $\cl$ across $\fc$ 
would have to be simultaneously $0$ and $1$ on $c^1$.
Therefore every facet of $T$ must be a union of facets of $T^{1}$.
Analogous to the proof of \propref{prop:ext},
the claim follows by considering ever coarser-scaled copies of the $\Tbase_j$
and hence of $T$.
\end{proof}

\exref{ex:overlap} illustrating \propref{prop:overlap}
shows that \propref{prop:family} does not
yield a sufficient constraint.
In generalizing \propref{prop:obtuse} to \oc\ spline
families, we restrict attention to families
of tessellations that minimize facet overlap.
\begin{definition}[efficient family of tessellations] 
A family of tessellations $\{T_j\}_{j=0..J}, T_j \in \R^\Dim$ 
is \emph{efficient} if the intersection of more than two cell facets 
is of co-dimension greater than $1$. 
\end{definition}

\begin{figure}[h]
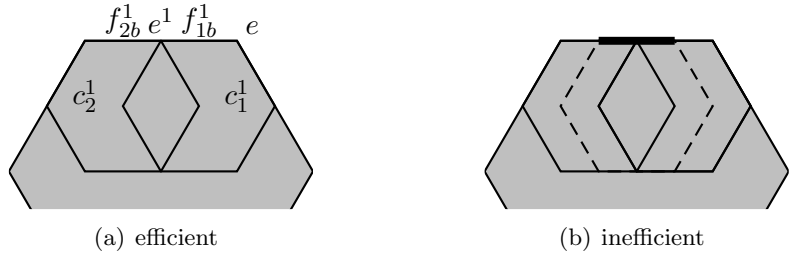

\centering
\subfigure[efficient] {
\psset{unit=1.0cm} \pspicture(-3,-0.5)(3,2)
\begin{psclip}{\pspolygon[linecolor=white](-2,2.3)(-2,-.5)(2,-.5)(2,2.3)}
   \pspolygon[fillstyle=solid,fillcolor=lightgray](2,0)(1,-1.732)
   (-1,-1.732)(-2,0)(-1,1.732)(1,1.732)
   \pspolygon[fillstyle=none](.5,0.866)(0,0)(-1,0)(-1.5,.866)(-1,1.732)(0,1.732)
   \pspolygon[fillstyle=none](1.5,0.866)(1,0)(0,0)(-0.5,.866)(0,1.732)(1,1.732)
   \rput(-0.5,2.0){$\fc^1_{2b}$}
   \rput( 0.5,2.0){$\fc^1_{1b}$}
   \rput(0.0,2.0){$\ed^1$}
   \rput(1.2,1.9){$\ed$}
   \rput(-1.0,1.0){$\supH^1_{2}$}
   \rput( 1.0,1.0){$\supH^1_{1}$}
   \end{psclip}
   \endpspicture
}
\subfigure[inefficient] {
\psset{unit=1.0cm} \pspicture(-3,-0.5)(3,2)
\begin{psclip}{\pspolygon[linecolor=white](-2,2.3)(-2,-.5)(2,-.5)(2,2.3)}
   \pspolygon[fillstyle=solid,fillcolor=lightgray](2,0)(1,-1.732)(-1,-1.732)(-2,0)(-1,1.732)(1,1.732)
   \pspolygon[fillstyle=none](.5,0.866)(0,0)(-1,0)(-1.5,.866)(-1,1.732)(0,1.732)
   \pspolygon[fillstyle=none,linestyle=dashed](1.0,0.866)(0.5,0)(-0.5,0)(-1,.866)(-0.5,1.732)(0.5,1.732)
   \pspolygon[fillstyle=none](1.5,0.866)(1,0)(0,0)(-0.5,.866)(0,1.732)(1,1.732)
   \pspolygon[fillstyle=none](1.5,0.866)(1,0)(0,0)(-0.5,.866)(0,1.732)(1,1.732)
   \psline[linewidth=3.0pt](-.5,1.732)(0.5,1.732) 
   \end{psclip}
   \endpspicture
}
\caption{ (a) Two \sd\ shifts of a hexagon cover the top facet (edge)
of the original hexagon exactly once from inside (gray region)
and hence twice if we continue the tessellation by reflection across the 
facet. Adding the dashed \sd\ hexagon in (b) 
covers a part of the top (thick edge) twice from inside so 
that reflection yields an inefficient family of tessellations.
}
\label{fig:efficient}
\end{figure}

\begin{prop}
\label{prop:overlap}
Let $\{T_j\}_{j=0..J}$ be an efficient family of tessellations of $\R^\Dim$ 
by shifts of one polyhedral cell $\cl \in T_0$.
If two facets $\fc_a$ and $\fc_b$ of $\cl$ meet with an obtuse angle $\ga$
and no facet of $\cl$ meets $\fc_b$ with an angle less than $\pi-\ga$,
and if $\clp$, the reflection of $\cl$ across $\fc_b$, is a cell of $T_0$
then $\bigcup_{j=0..J}\Chi(\Tbase_j)$ is not refinable.
\end{prop}
\begin{proof}
Assume $\bigcup_{j=0..J}\Chi(\Tbase_j)$ is  refinable.
Let $\ed$ be the $n-2$ dimensional intersection of two facets
$\fc_a$ and $\fc_b$ of the cell $\cl$.
By \propref{prop:family}, there exist 
$\supH^1_1, \supH^1_2 \in \{T^1_j\}_{j=0..J}$ 
whose two facets $\fc^1_{1b}$ and $\fc^1_{2b}$ lie on $\fc_b$,
whose shared boundary $\ed^1$ is parallel to $\ed$ (for $\Dim=2$,
$\ed^1$ is a point) and that lie both to the same side of $\fc_b$
as $\supH$. Without loss of generality,
$\fc^1_{1b}$ is closer to $\ed$ than is $\fc^1_{2b}$ 
(cf.\ \figref{fig:efficient}a).
When $n=2$, we set $\pt=\ed^1$ and when $n>2$, we pick a point $\pt$
in the interior of $\ed^1$. 

Due to efficiency, of all cells in $\{T^1_j\}_{j=0..J}$
that have one facet on $\fc_b$ and lie to the same side as $\supH$,
exactly three pieces
of the superposition $T^1$ of $\{T^1_j\}_{j=0..J}$
meet at $\pt$:
the piece $i_1$ solely inside $\supH^1_1$, 
the piece $i_2$ solely inside $\supH^1_2$, and 
$i_\cap$, the $\Dim$-dimensional intersection of $\supH^1_1$ and $\supH^1_2$.
The intersection $i_\cap$ exists and is $\Dim$-dimensional due to 
the obtuse angle $\ga $ of $\cl$ at $\ed$ and hence of $\supH^1_2$ at $\ed^1$;
and because $\cl$ and hence $\supH^1_1$ forms no angle less than $\pi-\ga$
with $\fc_b$. By reflection across $\fc_b$, 
there are three analogous pieces $o_j$ outside $\supH$ at $\pt$.
\figref{fig:overlap}a illustrates the situation in
the 2-dimensional plane $\pln$ through $\pt$ and orthogonal to $\ed^1$.

\begin{figure}[!ht]
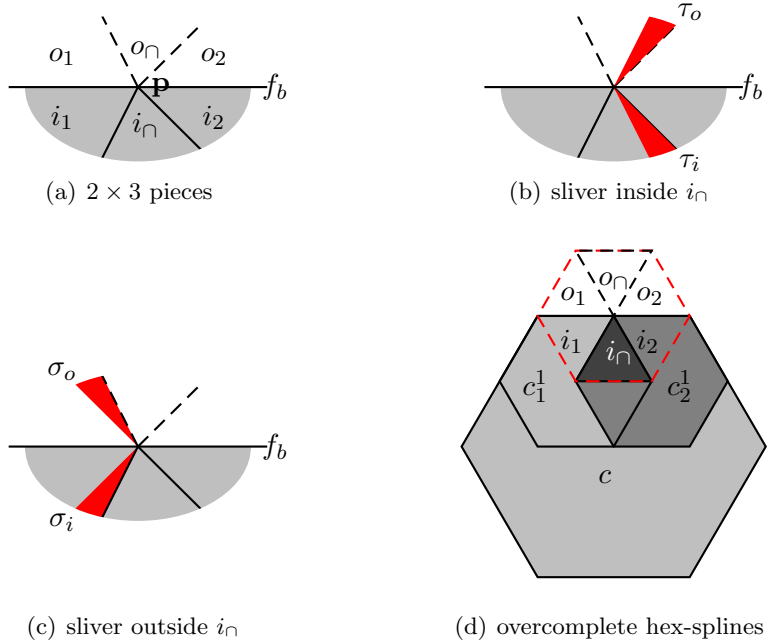

\centering
\subfigure[$2\times 3$ pieces ]{
\psset{unit=1.0cm} \pspicture(-3,-1)(3,1)
\begin{psclip}{\psellipse[linecolor=white](0,0)(1.5,1.0)}
\pspolygon*[linecolor=lightgray] (-1.5,-1)(-1.5,0)(1.5,0)(1.5,-1)
\psline[linestyle=dashed] (1,1)
\psline[linestyle=dashed] (-0.5,1)
\psline (1,-1)
\psline (-0.5,-1)
\end{psclip}
\psline (1.7,0)\psline(-1.7,0)  
\rput( 1.8,-0.0){$\fc_b$}
\rput( 0.3,-0.0){$\pt$}
\rput( 0.1,-0.5){$i_\cap$}
\rput(-1.0,-0.4){$i_1$}
\rput( 1.0,-0.4){$i_2$}
\rput( 0.1, 0.5){$o_\cap$}
\rput(-1.0, 0.4){$o_1$}
\rput( 1.0, 0.4){$o_2$}
\endpspicture
}
\subfigure[sliver inside $i_\cap$]{
\psset{unit=1cm} \pspicture(-3,-1)(3,1)
\begin{psclip}{\psellipse[linecolor=white](0,0)(1.5,1.0)}
\pspolygon*[linecolor=lightgray] (-1.5,-1)(-1.5,0)(1.5,0)(1.5,-1)
\psline[linestyle=dashed] (1,1)
\pspolygon*[linecolor=red] (0.5,1) (1,1) (0,0) 
\psline[linestyle=dashed] (-0.5,1)
\psline (1,-1)
\pspolygon*[linecolor=red] (0.5,-1) (1,-1) (0,0) 
\psline (-0.5,-1)
\end{psclip}
\psline (1.7,0)\psline(-1.7,0)  
\rput( 1.8,-0.0){$\fc_b$}
\rput( 1.0,-1.0){$\gt_i$}
\rput( 1.0, 1.0){$\gt_o$}
\endpspicture
}
\subfigure[sliver outside $i_\cap$]{
\psset{unit=1cm} \pspicture(-3,-2)(3,3)
\begin{psclip}{\psellipse[linecolor=white](0,0)(1.5,1.0)}
\pspolygon*[linecolor=lightgray] (-1.5,-1)(-1.5,0)(1.5,0)(1.5,-1)
\psline[linestyle=dashed] (1,1)
\pspolygon*[linecolor=red] (-1,1) (-.5,1) (0,0) 
\psline[linestyle=dashed] (-0.5,1)
\psline (1,-1)
\pspolygon*[linecolor=red] (-1,-1) (-.5,-1) (0,0) 
\psline (-0.5,-1)
\end{psclip}
\psline (1.7,0)\psline(-1.7,0)  
\rput( 1.8,-0.0){$\fc_b$}
\rput(-1.0,-1.0){$\gs_i$}
\rput(-1.0, 1.0){$\gs_o$}
\endpspicture
}
\subfigure[overcomplete hex-splines]{
\psset{unit=1.0cm} \pspicture(-3,-2)(3,3)
\pspolygon[fillstyle=solid,fillcolor=lightgray](2,0)(1,-1.732)(-1,-1.732)(-2,0)(-1,1.732)(1,1.732)
\rput[l](-0.2,-0.4){$\supH$}
\pspolygon[fillstyle=solid,fillcolor=gray](1.5,0.866)(1,0)(0,0)(-0.5,.866)(0,1.732)(1,1.732)
\pspolygon[fillstyle=none](.5,0.866)(0,0)(-1,0)(-1.5,.866)(-1,1.732)(0,1.732)
\pspolygon[fillstyle=solid,fillcolor=darkgray](0,1.732)(-.5,0.866)(.5,0.866)
\rput[l](.7,0.8){$\supH^1_2$}
\rput[l](-1.2,0.8){$\supH^1_1$}
\pspolygon[linestyle=dashed,linecolor=red,fillstyle=none](0.5,2.598)(-0.5,2.598)(-1,1.732)(-0.5,0.866)(0.5,0.866)(1,1.732)
\pspolygon[linestyle=dashed,fillstyle=none](0.5,2.598)(-0.5,2.598)(0,1.732)
\rput[l](-0.2,2.2){$o_\cap$}
\rput[l](-0.1,1.25){$\textcolor{white}{i_\cap}$}
\rput[l](0.3,1.4){$i_2$}
\rput[l](-0.7,1.4){$i_1$}
\rput[l](0.3,2.0){$o_2$}
\rput[l](-0.7,2.0){$o_1$}
\endpspicture
}
\caption{ Pieces (a) and slivers (red in b,c) in the plane $\pln$.
(d) Concrete partition of \exref{ex:overlap}
into pieces (sectors) of a neighborhood
of point $\pt$ on the partition of a facet $\fc_b$
of the support cell $\supH$ of $H$.
}
\label{fig:overlap}
\end{figure}
 
Denote by $\spl$ the spline formed as a linear combination of \ifn s
of the cells in $\{T^1_j\}_{j=0..J}$ that have one facet on $\fc_b$. 
To replicate the \ifn\ of $\cl$,
$\spl$ must have a unit step across $\fc_b$ and hence
$\spl(i_1) - \spl(o_1) = 1$ and $\spl(i_2) - \spl(o_2) = 1$
where $\spl$ applied to a piece of $T^1$
means evaluating $\spl$ at some point in that piece,
sufficiently close to $\pt$. Due to the overlap of the \ifn s 
on $i_\cap$ and on $o_\cap$, at $\pt$ 
\begin{equation}
   \spl(i_\cap)-\spl(o_\cap) 
   =
   \spl(i_1) + \spl(i_2) - \spl(o_1) - \spl(o_2)
   = 2,
\label{eq:two}
\end{equation}
incompatible with the unit step across $\fc_b$ at $\pt$.


Additional cells of $\{T^1_j\}$
that overlap all six pieces, 
$i_1$, $o_1$, 
$i_\cap$, $o_\cap$, 
$i_2$, $o_2$, 
surrounding $\pt$
do not affect the above difference since their \ifn s are constant.
It remains to consider cells with facets crossing $\pt$ and
it suffices to consider the 2-dimensional plane $\pln$ through $\pt$ 
and orthogonal to $\ed^1$.
In $\pln$, as shown in \figref{fig:overlap}a,b,c,
let the boundary ($i_\cap,i_2$) between $i_\cap$ and $i_2$ 
form a smaller-or-equal angle 
with $\fc_b$ than the boundary ($i_\cap,i_1$).
Due to efficiency, within $\cl$, crossing facets lie either
(b) strictly inside $i_\cap$ or (c) strictly outside $i_\cap$
since the boundaries must not be covered a third time.
In case (b) this strictness implies the existence of a sliver $\gt_i$,
i.e.\ a piece of $T^1$ devoid of crossing facets, 
attached to ($i_\cap,i_2$) and inside $i_\cap$.
Any linear combination of \ifn s corresponding to cells with 
crossing facets of type (b) therefore leaves unchanged the difference in value
and hence the incompatibility \eqref{eq:two}, 
between $\gt_i$ and its reflection $\gt_o$ across $\fc_b$
(see \figref{fig:overlap}b).
%
In case (c), 
any linear combination of \ifn s that modifies the value in $i_\cap$
also modifies the difference in value between two slivers $\gs_i$ and $\gs_o$,
where $\gs_i$ is attached to $(i_\cap,i_1)$ and lies
inside $i_1$ and $\gs_o$ is its reflection across $\fc_b$.
Specifically, when the value of $\spl(i_\cap) - \spl(o_\cap)$ is changed to 
satisfy the step condition, $\spl(\gs_i) - \spl(\gs_o)$ is changed 
away from the prescribed value 1.
Together this contradicts the assumption 
that $\bigcup_{j=0..J}\Chi(\Tbase_j)$ is  refinable.
\end{proof}

The following \exref{ex:overlap} illustrates \propref{prop:overlap}.
 
\begin{example}
\label{ex:overlap}
{\rm 
Consider, as in \figref{fig:overlap}d, shifts 
\begin{equation*}
   H_1(x) :=  H(x-\smt{-c}{s}), \quad
   H_2(x) :=  H(x-\smt{c}{s}), \quad
   c := \cos \frac{\pi}{3},\
   s := \sin \frac{\pi}{3}
\end{equation*}
of the \ifn\  $H(x)$ of a tessellation $T_0$.
The three corresponding tessellations intersect only in single triangles
and their superposition
contains no piece that \straddl es the support hexagon of $H$.
Yet this minimal family\footnote{
$J=2$ yields the minimal number of scaled families that can satisfy  the 
necessary constraints of \propref{prop:family} for hex-splines for 
three reasons.
First, scaling should at least be binary.
Second, the family that includes $H/2$ does not contribute to the step function
because its cells \straddl e the boundary of $\supH$ if they include 
a part of the boundary (see the dashed hexagon in \figref{fig:overlap}d).
Third, at least two $\hlf$-scaled edges,
namely of elements of $T^1_1$ and $T^1_2$, are required to cover any
edge of $H$.} is, as expected, not refinable.
\figref{fig:overlap}b shows the regions $i_\cap$ and $o_\cap$
where, for any spline $\spl$, $\spl(i_\cap) - \spl(o_\cap) = 2$ 
since 
$\spl(i_1) - \spl(o_1) = 1$ and $\spl(i_2) - \spl(o_2) = 1$.
\hfill $\Box$
}
\end{example}



The proof of \corref{cor:voronoi} showed that a pair of facets of 
the Voronoi cells of non-Cartesian crystallographic lattices 
form obtuse angles. By the symmetry of the cells, all facet angles are
obtuse and hence satisfy the angle criteria of \propref{prop:overlap}.
Together with the reflection symmetry of the lattices, \propref{prop:overlap}
implies the following generalization of \corref{cor:voronoi}.

\begin{corollary}
\label{cor:overVoronoi}
Efficient \oc\ families of splines obtained by convolution of 
the Voronoi cell of a non-Cartesian crystallographic root lattice
are not refinable.
\end{corollary}

\section{Conclusion}
\label{sec:conclusion}
The paper identified several necessary criteria for tessellations
to admit a refinable space of (convolutions of) \ifn s.
The criteria are chosen for their simplicity.
For example, we showed that admissible \sit s must contain, for every facet, 
the whole plane through that facet.
Already for hex-splines and hex-spline superpositions,
an alternative algebraic proof of non-refinability 
is considerably more involved.

\corref{cor:voronoi} and \ref{cor:overVoronoi}
show that the increased isotropy of the Voronoi cells of 
non-Cartesian root lattices prevents refinability,
even for \oc\ spaces obtained by efficient superposition 
of shifted lattices. 
Increased isotropy of the Voronoi cells is however
the main reason for considering non-Cartesian lattices in the first place:
these lattices have high packing densities that can 
improve sampling efficiency \cite{IC::PetersenM1962}.

In conclusion, if we seek shift-invariant 
refinable classes of splines from convolving \ifn s of polyhedral cells, 
remarkably few options 
exist apart from tensor-product B-splines and box-splines.
This does not imply that more general lattices fail to 
have associated refinable splines that represent their symmetry 
and translational structure. 
In the bivariate setting, odd orders of continuity on the hexagonal dual
of the regular triangulation can be filled in
by half-box splines \cite{Prautzsch:BS:2002}. 
And if fractal support is acceptable,
\cite{journals/cagd/OswaldS03,Han:I:2002} provide refinable 
functions with approximately hexagonal footprint.
Combining families of symmetric box-splines, such as 
\cite{Kim:2011:SBS} yields refinable splines for any level of smoothness
and crystallographic structure. 
It is just the particular approach of 
convolving non-Cartesian lattice Voronoi cells that fails to provide
the important spline property of refinability.

\medskip
{\bf Acknowledgement } 
Zhangjin Huang kindly worked out a first, algebraic proof of non-refinability 
for \exref{ex:hexref}, a scenario I posed to him. Andrew Vince and 
Carl de Boor helped me clarify the exposition in its early stages.

\small{
\bibliographystyle{alpha}
\bibliography{p}
}

\end{document}


\begin{pspicture}(-0.5,-1)(4,4)
  \psaxes(4,4)
    \begin{psclip}{%
        \pscustom[linecolor=red]{%
               \psplot{0.5}{4}{2 x div}
                      \lineto(4,4)}
                             \pscustom[linecolor=blue]{%
                                      \psplot{0}{3}{3 x x mul 3 div sub}
                                               \lineto(0,0)%
                                                      }%
                                                        }
                                                          \psframe*[linecolor=lightgray](4,4)
                                                            \end{psclip}%
                                                            \end{pspicture}